\documentclass[12pt]{article}
\usepackage{amssymb, enumerate}
\usepackage{latexsym}
\usepackage{url}

\usepackage[all,cmtip]{xy}
\usepackage{amsmath,amssymb}

\newtheorem{theorem}{Theorem}[section]
\newtheorem{lemma}[theorem]{Lemma}
\newtheorem{corollary}[theorem]{Corollary}
\newtheorem{proposition}[theorem]{Proposition}
\newtheorem{remark}[theorem]{Remark}
\newtheorem{remarks}[theorem]{Remarks}

\newenvironment{proof}{\normalsize {\sc Proof}.}{{\hfill $\Box$%
 \hskip - \parfillskip\bigskip}}

\newcommand{\Syl}{\mathop{\rm Syl}\nolimits}
\newcommand{\Irr}{\mathop{\rm Irr}\nolimits}
\newcommand{\IBr}{\mathop{\rm IBr}\nolimits}

\newcommand{\Aut}{\mathop{\rm Aut}\nolimits}
\newcommand{\Out}{\mathop{\rm Out}\nolimits}

\newcommand{\Pic}{\mathop{\rm Pic}\nolimits}

\newcommand{\NN} {\mathbb{N}}

\newcommand{\cO} {\mathcal{O}}

\def\L{\Lambda}

\def\bigcp{\mathop{\mathchoice 
 {\hbox{\sf\Large\lower 0.1\baselineskip\hbox{Y}}}%
 {\hbox{\sf\large\lower 0.1\baselineskip\hbox{Y}}}%
 {\hbox{\sf\normalsize\lower 0.1\baselineskip\hbox{Y}}}%
 {\hbox{\sf\tiny\lower 0.1\baselineskip\hbox{Y}}}%
}}
\def\bigtimes{\mathop{\mathchoice 
 {\hbox{\sf\Large\lower 0.1\baselineskip\hbox{X}}}%
 {\hbox{\sf\large\lower 0.1\baselineskip\hbox{X}}}%
 {\hbox{\sf\normalsize\lower 0.1\baselineskip\hbox{X}}}%
 {\hbox{\sf\tiny\lower 0.1\baselineskip\hbox{X}}}%
}}

\def\Sym(#1){\mathop{\rm Sym}(#1)}
\def\Sym(#1){S_{#1}}
\def\diag(#1){\mathop{\rm diag}(#1)}

\newenvironment{enumerate*}{%
 \begin{enumerate}%
 }%
 {\end{enumerate}}

\begin{document}

\title{Morita equivalence classes of blocks with elementary abelian defect groups of order $16$ \footnote{This research was supported by the EPSRC (grant no. EP/M015548/1).}}

\author{Charles W. Eaton\footnote{School of Mathematics, University of Manchester, Manchester, M13 9PL, United Kingdom. Email: charles.eaton@manchester.ac.uk}}

\date{15th May 2019}
\maketitle


\begin{abstract}
We classify the Morita equivalence classes of blocks with elementary abelian defect groups of order $16$ with respect to a complete discrete valuation ring with algebraically closed residue field of characteristic two. As a consequence, blocks with this defect group are derived equivalent to their Brauer correspondent in the normalizer of a defect group and so satisfy Brou\'e's Conjecture.

\end{abstract}


\section{Introduction}

Throughout let $k$ be an algebraically closed field of prime characteristic $\ell$ and let $\cO$ be a discrete valuation ring with residue field $k$ and field of fractions $K$ of characteristic zero. We assume that $K$ is large enough for the groups under consideration. We consider blocks $B$ of $\mathcal{O}G$ with defect group $D$, for finite groups $G$.

Our purpose is the description of the Morita and derived equivalence classes of (module categories for) blocks of finite groups with a given defect group. It is already known by~\cite{ekks14} that Donovan's conjecture holds for elementary abelian $2$-groups, that is, for each $n \in \mathbb{N}$ there are only finitely many Morita equivalence classes of blocks with defect group $(C_2)^n$, and so in theory Morita equivalence classes of such blocks could be classified for any given $n$. Here we consider the case $n=4$ and achieve a complete classification. The main tool is the description given in~\cite{ekks14} of the $2$-blocks with abelian defect groups of the quasisimple groups. The number of irreducible ordinary and Brauer characters of blocks with defect group $(C_2)^4$ has already been determined in~\cite{ks13} and~\cite{ekks14}. Our work continues~\cite{ea16} in which a classification is given for blocks with elementary abelian defect groups of order $8$. The Morita equivalence classes of block with Klein four defect groups are known by~\cite{er82} and~\cite{li94}. Other $\ell$-groups where there are classifications are: cyclic $\ell$-groups, where the Morita equivalence classes can be characterised in terms of Brauer trees (in work by many, for which see~\cite{li96}); abelian $2$-groups of $2$-rank at most three (see~\cite{ekks14},~\cite{el17} and~\cite{wzz17}); dihedral, semidihedral and generalized quaternion $2$-groups except where the block has two simple modules in the case that the defect group is generalized quaternion (see~\cite{er90}, and note that this classification is only known with respect to $k$); nonmetacylic minimal nonabelian $2$-groups $\langle x,y: x^{2^r}=y^{2^s}=[x,y]^2=[x,[x,y]]=[y,[x,y]]=1\rangle$, where $r \geq s \geq 1$, by~\cite{sa11} and~\cite{eks12}; $2$-groups which are a direct product of cyclic factors all of different orders, whose automorphism groups are themselves $2$-groups so the block must be nilpotent and so Morita equivalent to $\cO D$ by~\cite{BP80} and~\cite{pu88}; the remaining metacyclic $2$-groups not listed above, which force the block to be nilpotent by~\cite{cg12}. Finally, principal blocks with defect group $C_3 \times C_3$ are classified (even up to Puig equivalence) in~\cite{ko03}.

A significant challenge arises in our situation which does not arise for defect groups of order $8$ in that we must address Morita equivalences in the case of a normal subgroup of index $3$ where there are infinitely many possibilities for $N$ but the Morita equivalence class of the block of $N$ is fixed. We do so by following K\"{u}lshammer's analysis in~\cite{ku95} of possible crossed products in detail.

Before stating the main theorem, we recall the definition of a subpair and of the inertial quotient of $B$.

A $B$-subpair is a pair $(Q,b_Q)$ where $Q$ is a $p$-subgroup of $G$ and $b_Q$ is a block of $\cO QC_G(Q)$ with Brauer correspondent $B$. When $D$ is a defect group for $B$, the $B$-subpairs $(D,b_D)$ are $G$-conjugate. Write $N_G(D,b_D)$ for the stabilizer in $N_G(D)$ of $b_D$. Then the \emph{inertial quotient} of $B$ is $E=N_G(D,b_D)/DC_G(D)$, an $\ell'$-group unique up to isomorphism. To simplify further definitions, suppose that $D$ is abelian. We say $(Q,b_Q) \leq (R,b_R)$ for subpairs $(Q,b_Q)$ and $(R,b_R)$ if $Q \leq R$ and $(b_R)^{C_G(Q)}=b_Q$. If $x \in D$ and $b_x$ is a block of $C_G(x)$ we say that the Brauer element $(x,b_x) \in (D,b_D)$ if $(b_D)^{C_G(x)}=b_x$.

Note that if $D$ is abelian, then $B$ is nilpotent precisely when the inertial quotient is trivial.

The possible inertial quotients for a block with defect group $D$ are given in~\cite{ks13}, and these are: $1$ (corresponding to nilpotent blocks); $C_3$ with action as in $A_4 \times (C_2)^2$; $C_3$ consisting of $5$th powers of a Singer cycle for $\mathbb{F}_{16}$ (with only one fixed point in its action on $D$); $C_5$ consisting of $3$rd powers of a Singer cycle; $C_7$ coming from a Singer cycle for $\mathbb{F}_{8}$; $C_3 \times C_3$; $C_7 \rtimes C_3$ coming from a Singer cycle and a field automorphism of $\mathbb{F}_8$; $C_{15}$ coming from a Singer cycle of $\mathbb{F}_{16}$. Each inertial quotient apart from $C_3 \times C_3$ has trivial Schur multiplier, and in this case it is $C_3$.

We say that a block with defect group $(C_2)^4$ is of type $E$ if it has inertial quotient $E$ where there is only one possible faithful action on $(C_2)^4$, and in the case that the inertial quotient is $C_3$, we say it has type $(C_3)_1$ when the action is as in $A_4 \times (C_2)^2$ and type $(C_3)_2$ when there is only one fixed point.

\begin{theorem}
\label{maintheorem}
Let $B$ be a block of $\cO G$ with elementary abelian defect group $D$ of order $16$, where $G$ is a finite group. Then $B$ is Morita equivalent to precisely one of the following:

(a) a non-principal block of $(C_2)^4 \rtimes 3_+^{1+2}$, where the centre of $3_+^{1+2}$ acts trivially, and we note that the two non-principal blocks are Morita equivalent;

(b) the principal block of precisely one of the following:

(i) $D$;

(ii) $(C_2)^2 \times A_4$;

(iii) $(C_2)^2 \times A_5$;

(iv) $D \rtimes C_3$ of type $(C_3)_2$;

(v) $D \rtimes C_5$;

(vi) $C_2 \times ((C_2)^3 \rtimes C_7)$;

(vii) $C_2 \times SL_2(8)$ (type $C_7$);

(viii) $A_4 \times A_4$;

(ix) $A_4 \times A_5$ (type $C_3 \times C_3$);

(x) $A_5 \times A_5$ (type $C_3 \times C_3$);

(xi) $D \rtimes C_{15}$;

(xii) $SL_2(16)$ (type $C_{15}$);

(xiii) $C_2 \times ((C_2)^3 \rtimes (C_7 \rtimes C_3))$;

(xiv) $C_2 \times J_1$ (type $C_7 \rtimes C_3$);

(xv) $C_2 \times \Aut(SL_2(8))$ (type $C_7 \rtimes C_3$).

\medskip

If $B$ is a principal block, then it is Morita equivalent to one of the examples in case (b), i.e., the blocks in case (a) cannot be Morita equivalent to a principal block of any finite group.

\medskip

Blocks are derived equivalent if and only if they have the same inertial quotient and number of simple modules.
\end{theorem}

\begin{remarks}
(i) It will be clear from the proof of Theorem \ref{maintheorem} that blocks with defect group $(C_2)^4$ cannot be Morita equivalent to a block with non-isomorphic defect group. The same can be said of the inertial quotient, as the number of irreducible characters and number of simple modules together determine the inertial quotient.

(ii) Non-nilpotent blocks with defect group $(C_2)^4$ and just one simple module are studied in~\cite{ls16}. The structure of the centre of such a block is described there, and hopefully in the future there will be a classification-free proof that all such blocks are Morita equivalent to the blocks in part (a) of Theorem \ref{maintheorem}.
\end{remarks}

\begin{corollary}
Broue's abelian defect group conjecture holds for blocks with defect group $(C_2)^4$, that is, if $B$ is a block of $\cO G$ with such a defect group $D$, then $B$ is derived equivalent to its Brauer correspondent in $\cO N_G(D)$.
\end{corollary}

\begin{remark}
We cannot say at present whether the blocks are splendid equivalent as we cannot say anything about the sources of the bimodules giving the Morita equivalences in Theorem \ref{maintheorem}.
\end{remark}

The paper is structured as follows. In Section \ref{prelim} we give many of the miscellaneous preliminary results necessary for the proof of Theorem \ref{maintheorem}. In Section \ref{crossedproducts} we give the necessary background on crossed products and Picard groups, and analyse possible crossed products in situations that will arise in the proof of the main theorem. In Section \ref{proofofmaintheorem} we prove Theorem \ref{maintheorem}.


\section{Preliminary results}
\label{prelim}

Let $G$ be a finite group, $N \lhd G$ and let $b$ be a $G$-stable block of $\cO N$. The normal subgroup $G[b]$ of $G$ is defined to be the group of elements of $G$ acting as inner automorphisms on $b \otimes_{\cO} k$. We first collect some results concerning $G[b]$ that will be used when considering automorphism groups of simple groups.

\begin{proposition}
\label{innerautos}
Let $G$ be a finite group and $B$ a block of $\cO G$ with defect group $D$. Let $N \lhd G$ with $D \leq N$ and suppose that $B$ covers a $G$-stable block $b$ of $\cO N$. Let $B'$ be a block of $\cO G[b]$ covered by $B$. Then

(i) $b$ is source algebra equivalent to $B'$, and in particular has isomorphic inertial quotient;

(ii) $B$ is the unique block of $\cO G$ covering $B'$.
\end{proposition}

\begin{proof}
Part (i) is~\cite[2.2]{kkl12}, noting that a source algebra equivalence over $k$ implies one over $\cO$ by~\cite[7.8]{pu88}. Part (ii) follows from~\cite[3.5]{da73}.
\end{proof}

The following is a distillation of those results in~\cite{kp90} which are relevant here.

\begin{proposition}[\cite{kp90}]
\label{kp}
Let $G$ be a finite group and $N \lhd G$. Let $B$ be a block of $\cO G$ with defect group $D$ covering a $G$-stable nilpotent block $b$ of $\cO N$ with defect group $D \cap N$. Then there is a finite group $L$ and $M \lhd L$ such that (i) $M \cong D \cap N$, (ii) $L/M \cong G/N$, (iii) there is a subgroup $D_L$ of $L$ with $D_L \cong D$ and $D_L \cap M \cong D \cap N$, and (iv) there is a central extension $\tilde{L}$ of $L$ by an $\ell'$-group, and  a
 block $\tilde{B}$ of  $\cO \tilde L$ which is Morita equivalent to $B$ and has defect group $\tilde{D} \cong D_L \cong D$.

If $B$ is the principal block, then $\tilde{B}$ is the principal block.
\end{proposition}

\begin{proof}
Guidance on the extraction of these results form~\cite{kp90} is given in~\cite[2.2]{ea16}. It remains to prove the claim regarding the principal block.  Note that if $B$ is the principal block, then $b$ is also principal and so $N$ has a normal $\ell$-complement. Then $O_{\ell'}(N)$ lies in the kernel of $B$ and the corresponding $\tilde{B}$ is the principal block.
\end{proof}

Recall that a block of a finite group $G$ is \emph{quasiprimitive} if every block of every normal subgroup that it covers is $G$-stable under conjugation.

\begin{corollary}
\label{kpcor}
Let $G$ be a finite group and $N \lhd G$ with $N \not\leq Z(G)O_{\ell}(G)$. Let $B$ be a quasiprimitive block of $\cO G$ with defect group $D$ covering a nilpotent block $b$ of $\cO N$. Then there is a finite group $H$ with $[H:O_{\ell'}(Z(H))] <  [G:O_{\ell'}(Z(G))]$ and a block $B_H$ with defect group $D_H \cong D$ such that $B_H$ is Morita equivalent to $B$.
\end{corollary}

\begin{proof}
Let $b'$ be the block of $\cO Z(G)N$ covered by $B$ and covering $b$. Then $b'$ must also be nilpotent, and we may assume that $Z(G) \leq N$. Applying Proposition \ref{kp}, we may take $H=\tilde{L}$ and $B_H=\tilde{B}$. Note that $[\tilde{L}:O_{\ell'}(Z(\tilde{L}))] \leq |L| = [G:N]|D \cap N| < [G:O_{\ell'}(Z(G))]$.
\end{proof}

\begin{lemma}
\label{solvablequotient}
Let $G$ be a finite group and let $N \lhd G$ with $G/N$ $\ell$-solvable. Let $B$ be a quasiprimitive block of $G$ with abelian defect group $D$. Then $DN/N \in \Syl_\ell(G/N)$.
\end{lemma}

\begin{proof}
There is a series $N=G_0 \lhd G_1 \lhd \cdots \lhd G_n=G$ with each $G_i$ normal in $G$ and each quotient either an $\ell$-group or an $\ell'$-group. Let $B_i$ be the unique block of $\cO G_i$ covered by $B$. If $G_{i+1}/G_i$ is an $\ell'$-group, then $B_{i+1}$ and $B_i$ share a defect group. Suppose that $G_{i+1}/G_i$ is an $\ell$-group. Then $B_{i+1}$ is the unique block of $\cO G_{i+1}$ covering $B_i$, and so by~\cite[Theorem 15.1]{alp} a defect group $D_{i+1}$ of $B_{i+1}$ satisfies $D_{i+1}G_i/G_i=G_{i+1}/G_i$. The result follows.

\end{proof}

\begin{proposition}[\cite{wa94}]
\label{watanabe}
Let $B$ be an $\ell$-block of $\cO G$ for a finite group $G$ and let $Z \leq O_\ell(Z(G))$. Let $\bar B$ be the unique block of $\cO (G/Z)$ corresponding to $B$. Then $B$ is nilpotent if and only if $\bar B$ is nilpotent.
\end{proposition}

\begin{proof} The result in~\cite{wa94} is stated over $k$, but it follows over $\cO$ immediately.
\end{proof}

Recall that a block $B$ of $\cO G$ is \emph{nilpotent covered} if there is a finite group $H$ with $G \lhd H$ and a nilpotent block of $\cO H$ covering $B$. Let $D$ be a defect group for $B$ and let $b$ be the Brauer correspondent of $B$ in $\cO N_G(D)$. Following~\cite{pu11} $B$ is \emph{inertial} if it is basic Morita equivalent to $b$, that is, if there is a Morita equivalence induced by a bimodule with endopermutation source.

\begin{proposition}[\cite{pu11},~\cite{zh16}]
\label{inertial}
Let $G$ and $N$ be finite groups and $N \lhd G$. Let $b$ be a block of $\cO N$ covered by a block $B$ of $\cO G$.

(i) If $B$ is inertial, then $b$ is inertial.

(ii) If $b$ is nilpotent covered, then $b$ is inertial.

(iii) If $\ell \not| [G:N]$ and $b$ is inertial, then $B$ is inertial.

(iv) If $b$ is nilpotent covered, then it has abelian inertial quotient.
\end{proposition}

\begin{proof}
(i) is~\cite[Theorem 3.13]{pu11}, (ii) and (iv) are~\cite[Corollary 4.3]{pu11}. (iii) is the main theorem of~\cite{zh16}.
\end{proof}

We will make frequent use of the classification of Morita equivalence classes of blocks with Klein four defect groups throughout this paper without further reference:

\begin{proposition}[\cite{er82},~\cite{li94},~\cite{CEKL}]
Let $B$ be a block of $\cO G$ for a finite group $G$. If $B$ has Klein four defect group $D$, then it is source algebra equivalent to the principal block of one of $\cO D$, $\cO A_4$ and $\cO A_5$.
\end{proposition}

We extract the results of~\cite{ekks14} necessary for this paper:

\begin{proposition}[\cite{ekks14}]
\label{classification16}
Let $B$ be a block of $\cO G$ for a quasisimple group $G$ with elementary abelian defect group $D$ of order dividing $16$. Then one or more of the following occurs:

(i) $G \cong SL_2(16)$, $J_1$ or ${}^2G_2(q)$, where $q=3^{2m+1}$ for some $m \in \NN$, and $B$ is the principal block;

(ii) $G \cong Co_3$ and $B$ is the unique non-principal $2$-block of defect $3$;

(ii) $G$ is of type $D_n(q)$ or $E_7(q)$ for some $q$ of odd prime power order, $O_2(G)=1$ and $B$ is Morita equivalent to a block $C$ of a $\cO L$ where $L =L_0 \times L_1 \leq G$ such that $L_0$ is abelian and the block of $\cO L_1$ covered by $C$ has Klein four defect groups;

(iii) $|O_2(G)|=4$ and $D/O_2(G)$ is a Klein four group;

(iv) $B$ is nilpotent covered.
\end{proposition}

\begin{proof}
This follows from Proposition 5.3 and Theorem 6.1 of~\cite{ekks14}.
\end{proof}

One obstacle in classifying Morita equivalence classes over $\cO$ rather than $k$ is that the results of~\cite{kk96} only apply over $k$. However in our situation we are lucky to be able to apply some work of Watanabe on perfect isometries to obtain the same result over $\cO$ in certain crucial cases. For the benefit of the reader we state the relevant result of~\cite{wa00} here. First we need some more notation.

Let $B$ be a block of $\cO G$, where $G$ is a finite group. Write $\mathcal{L}_K(G,B)$ for the group of generalized characters of $B$ with respect to $K$. Let $\chi$ be a generalized character of $B$. Fix a maximal $B$-subpair $(D,b_D)$. Let $\lambda$ be a generalized character of a defect group $D$ of $B$ such that whenever $(x,b_x) \in (D,b_D)$ and $z \in G$ such that $(x,b_x)^z \in (D,b_D)$, we have $\lambda(x)=\lambda(x^z)$. Define $\lambda * \chi$ as in~\cite{bp80}, another generalized character of $B$. In the following, if $\lambda$ is a generalized character of a factor group of $D$, then we are implicitly considering its inflation to $D$.

\begin{proposition}[Lemma 3 of~\cite{wa00}]
\label{watanabe_isometry}
Let $B$ be a block of a finite group $G$ covering a $G$-stable block $b$ of $N \lhd G$. Suppose that $B$ has an abelian defect group $D$ and there is $Q \leq D$ such that $D = Q \times (D \cap N)$ and $G=N \rtimes Q$. Let $b_D$ be a block of $C_G(D)$ with Brauer correspondent $B$, and write $B'=(b_D)^{C_G(Q)}$. If there is a perfect isometry $I:\mathcal{L}_K(C_G(Q),B') \rightarrow \mathcal{L}_K(G,B)$ satisfying $I(\lambda * \zeta)=\lambda * I(\zeta)$ for all $\lambda \in \Irr(Q)$ and $\zeta \in \Irr(B')$, then $B \cong \cO Q \otimes_\cO b$ as $\cO$-algebras.
\end{proposition}

\begin{proposition}
\label{KK} Let $G$ be a finite group  and let  $B$ be a block of $\cO G$ with elementary abelian defect group $D$ of order $16$ and cyclic inertial quotient. Suppose $N \unlhd G$ with $G=ND$. If $B$ covers a non-nilpotent $G$-stable block $b$ of $\cO N$, then there is an elementary abelian $2$-group $Q \leq D$  with $G=N \rtimes Q$ such that $B$ is Morita equivalent to a block $C$  of  $\cO (N \times Q) $  with  defect group $(D \cap N) \times Q \cong D$.
\end{proposition}

\begin{proof} Let $b_D$ be a block of $C_G(D)$ with Brauer correspondent $B$, and write $E=N_G(D,b_D)/C_G(D)$ as described in the introduction. We may suppose $G \neq N$, so $|E| \leq 7$. By~\cite[Theorem 15]{sa17} we have $l(B)=|E|$.

Following~\cite{wa05}, we may write $D=D_1 \times D_2$ where $D_1=C_D(N_G(D,b_D))$ and $D_2 = [N_G(D,b_D),D]$. We have $D \rtimes E=D_1 \times (D_2 \rtimes E)$. Since $D_2 \leq N$, $E$ is cyclic and $D$ is elementary abelian, we may choose $Q$ to be a direct factor of $D_1$.

 By the main theorem of~\cite{wa05} there is a perfect isometry $$I:\mathcal{L}_K(N_G(D,b_D),b^{N_G(D,b_D)}) \rightarrow \mathcal{L}_K(G,B)$$ such that $I(\lambda * \zeta)=\lambda * I(\zeta)$ for all $\lambda \in \Irr(D_1)$ and $\zeta \in \mathcal{L}_K(N_G(D,b_D),b^{N_G(D,b_D)})$. We have $N_G(D,b_D) \leq C_G(D_1) \leq C_G(Q)$. Let $B'=(b_D)^{C_G(Q)}$. Now $B'$ also has inertial quotient $E$ and we may apply the same argument to obtain a perfect isometry
 $$J:\mathcal{L}_K(N_G(D,b_D),b^{N_G(D,b_D)}) \rightarrow \mathcal{L}_K(C_G(Q),B')$$ such that $J(\lambda * \zeta)=\lambda * J(\zeta)$ for all $\lambda \in \Irr(D_1)$ and $\zeta \in \mathcal{L}_K(N_G(D,b_D),b^{N_G(D,b_D)})$. We may then apply Proposition \ref{watanabe_isometry} to $I \circ J^{-1}$ and the result follows.
\end{proof}

In the above note that if $b$ is Morita equivalent to a block $c$ of $\cO M$ for some finite group $M$, then $C$ is Morita equivalent to the block $c \otimes \cO Q$ of $\cO (M \times Q)$.

\begin{lemma}
\label{typeC3orC3xC3}
Let $G$ be a finite group and $N \lhd G$ with $G/N$ of odd order (and solvable). Let $B$ be a block of $\cO G$ covering a $G$-stable block $b$ of $\cO N$ with defect group $D \cong (C_2)^4$. Suppose that $B$ covers no nilpotent block of any normal subgroup $M \lhd G$ with $N \leq M$. If $b$ is of type $C_3 \times C_3$ or $(C_3)_1$, then $B$ is also of one of these two types.
\end{lemma}

\begin{proof}
It suffices to consider the case that $[G:N]$ is an odd prime, say $w$. Note that $B$ and $b$ share the defect group $D$.

Suppose $C_G(D)=C_N(D)$. Then the inertial quotient of $B$ contains that of $b$ with index dividing $w$. Since $C_3 \times C_3$ is maximal amongst subgroups of odd order of $GL_4(2)$ and is the only subgroup containing $C_3$ as a normal subgroup the result follows in this case.

Suppose $C_G(D) \neq C_N(D)$. Let $(D,b_D)$ be a $b$-subpair and let $(D,B_D)$ be a $B$-subpair with $B_D$ covering $b_D$. If $C_G(D) \not\leq N_G(D,b_D)$, then $B_D$ covers $w$ conjugates of $b_D$ and $B_D$ is the unique block of $C_G(D)$ covering $b_D$, so $N_G(D,B_D) = C_G(D)N_G(D,b_D)$. Hence $N_G(D,B_D)/C_G(D) \cong N_N(D,b_D)/C_N(D)$ and we are done in this case. If $C_G(D) \leq N_G(D,b_D)$, then $N_G(D,B_D) \leq N_G(D,b_D)$ as $b_D$ is the unique block of $C_N(D)$ covered by $B_D$. Now $[N_G(D,B_D):C_G(D)]$ divides $[N_N(D,b_D):C_N(D)]$ and we are done.
\end{proof}



\section{Crossed products and Picard groups}
\label{crossedproducts}

An essential part of a reduction of Donovan's conjecture to quasisimple groups is K\"ulshammer's analysis in~\cite{ku95} of the situation of a normal subgroup containing the defect groups of a block, which involves the study of crossed products of a basic algebra with an $\ell'$-group. In the general setting he finds finiteness results for the possible crossed products, but in our situation we are able to precisely describe the possibilities using knowledge of the Picard groups of certain basic algebras.

Background on crossed products may be found in~\cite{ku95}, but we summarize what we need here. Let $X$ be a finite group and $R$ an $\cO$-algebra. A crossed product of $R$ with $X$ is an $X$-graded algebra $\L$ with identity component $\L_1 = R$ such that each graded component $\L_x$, where $x \in X$, contains a unit $u_x$. Given a choice of unit $u_x$ for each $x$, we have maps $\alpha: X \rightarrow \Aut (R)$ given by conjugation by $u_x$ and $\mu:X \times X \rightarrow U(R)$ given by $\alpha_x \circ \alpha_y = \iota_{\mu(x,y)} \circ \alpha_{xy}$, where $U(R)$ is the group of units of $R$ and $\iota_{\mu(x,y)}$ is conjugation by $\mu(x,y)$. The pair $(\alpha,\mu)$ is called a parameter set of $X$ in $R$. In~\cite{ku95} an isomorphism of crossed products respecting the grading is called a weak equivalence. By the discussion following Proposition 2 of~\cite{ku95} weak isomorphism classes of crossed products of $R$ with $X$ are in bijection with pairs consisting of an $\Out(R)$-conjugacy class of homomorphisms $X \rightarrow \Out(R)$ for which the induced element in $H^3(X,U(Z(R)))$ vanishes and an element of $H^2(X,U(Z(R)))$.

Note that $\alpha: X \rightarrow \Aut (R)$ restricts to a map $X \rightarrow \Aut (Z(R))$, which makes $Z(R)$ an $X$-algebra. The $k$-algebras $Z(R)/J(Z(R))$ and $U(Z(R)/J(Z(R)))$ also become $X$-algebras.

Now suppose that $X = \langle x \rangle$ is a cyclic $\ell'$-group. Following the strategy in~\cite[Section 3]{ku95}, $U(Z(R)) \cong U(Z(R)/J(Z(R))) \times (1+J(Z(R))$ and $$H^2(X,U(Z(R))) \cong H^2(X,U(Z(R)/J(Z(R)))) \times H^2(X,1+J(Z(R))).$$ We have $H^2(X,1+J(Z(R)))=0$ since $X$ is an $\ell'$-group. Now $Z(R)/J(Z(R))$ is a commutative semisimple $k$-algebra, which we denote $A$, and note as above that it is an $X$-algebra. Write $A=A_1 \times \cdots \times A_r$, where each $A_i$ is a product of simple algebras constituting an $X$-orbit. We have $H^2(X,U(A)) \cong H^2(X,U(A_1)) \times \cdots \times H^2(X,U(A_r))$. We claim each $H^2(X,U(A_i))$ vanishes. As a $kX$-module $U(A_i)$ is induced from the trivial module of $kY$ for some $Y \leq X$, and so by Shapiro's Lemma $H^2(X,U(A_i)) \cong H^2(Y,k^\times)$ (see~\cite[2.8.4]{ben1}), which vanishes since $X$ is cyclic. We have shown that $H^2(X,U(Z(R)))=0$ for each $i$.


Now suppose that we have a finite group $G$ and $N \lhd G$ with $G/N$ an $\ell'$-group. Suppose that $B$ is a block of $\cO G$ covering a $G$-stable block $b$ of $\cO N$. Define $X=G/G[b]$. Let $C$ be a block of $\cO G[b]$ covering $b$ (and covered by $B$), which by Proposition \ref{innerautos} is Morita equivalent to $b$. Let $f$ be an idempotent of $C$ such that $fCf$ is a basic algebra. Following~\cite{ku95}, which is performed over $\cO$ in~\cite{ei18}, we may consider $fBf$ as a crossed product of $fCf$ with $X$, and $fBf$ is Morita equivalent to $B$. 

The Picard group $\Pic(R)$ of $R$ consists of isomorphism classes $R$-$R$-bimodules which induce Morita self-equivalences of $b$. For $b$-$b$-bimodules $M$ and $N$, the group multiplication is given by $M \otimes_b N$. Let $\varphi \in \Aut(R)$. Define the $R$-$R$-bimodule ${}_\varphi R$ by letting ${}_\varphi R=R$ as sets and defining $a_1 \cdot m \cdot a_2 = \varphi(a_1)ma_2$ for $a_1,a_2,m \in b$. By~\cite[55.11]{cr2} inner automorphisms give isomorphic bimodules and $\varphi \mapsto {}_\varphi R$ gives rise to an injection $\Out(R) \rightarrow \Pic(R)$. If $R$ is a basic algebra, then this is an isomorphism. We direct the reader to~\cite{bkl17} for a thorough investigation of Picard groups of blocks with respect to discrete valuation rings.

For most of the blocks which appear as candidates for $b$ in this paper the Picard group is known by~\cite{el18}. We gather this information here:

\begin{proposition}[\cite{el18}]
\label{Picard}
Let $Q$ be a finite abelian $2$-group.

(i) $\Pic(\cO(A_4 \times Q)) \cong S_3 \times (Q \rtimes \Aut(Q))$.

(ii) $\Pic(B_0(\cO(A_5 \times Q))) \cong C_2 \times (Q \rtimes \Aut(Q))$.

(iii) $\Pic(\cO(A_4 \times A_4)) \cong S_3 \wr C_2$.

(iv) $\Pic(B_0(\cO(A_5 \times A_4))) \cong S_3 \times C_2$.
\end{proposition}

Applying all of the above, we have the following:

\begin{proposition}
\label{applying_cross_products}
Let $G$ be a finite group and $N \lhd G$ with $G/N$ cyclic of odd prime order. Let $b$ be a $G$-stable block of $\cO N$ with defect group $D \cong (C_2)^4$.

(i) If $b$ is Morita equivalent to $\cO (A_4 \times C_2 \times C_2)$, then $B$ is Morita equivalent to $b$, $\cO D$, $\cO(A_4 \times A_4)$ or a non-principal block of $\cO(C_2)^4 \rtimes 3_+^{1+2}$, where the centre of $3_+^{1+2}$ acts trivially.

(ii) If $b$ is Morita equivalent to the principal block of $\cO(A_5 \times C_2 \times C_2)$, then $B$ is Morita equivalent to $b$ or the principal block of $\cO(A_5 \times A_4)$.

(iii) If $b$ is Morita equivalent to $\cO(A_4 \times A_4)$, then $B$ is Morita equivalent to $b$, $\cO(A_4 \times C_2 \times C_2)$ or $\cO((C_2)^4 \rtimes C_3))$ where the $C_3$ acts with only one fixed point.

(iv) If $b$ is Morita equivalent to the principal block of $\cO(A_4 \times A_5)$, then $B$ is Morita equivalent to $b$ or the principal block of $\cO(C_2 \times C_2 \times A_5)$.

(v)  If $b$ is Morita equivalent to the principal block of $\cO(A_5 \times A_5)$, then $B$ is Morita equivalent to $b$.

\end{proposition}

\begin{proof}
Let $B$ be a block of $G$ covering $b$. Either $G[b]=G$ or $G[b]=N$. By Proposition \ref{innerautos} if $G[b]=G$, then $B$ is Morita equivalent to $b$, in which case we are done. Hence suppose $G[b]=N$, so by Proposition \ref{innerautos} $B$ is the unique block of $G$ covering $b$. Note that $B$ and $b$ share a defect group.

We treat cases (i)-(iv) first. Case (v) uses a different strategy since we do not know the Picard group for the principal block of $\cO(A_5 \times A_5)$.

Define $X=G/N$. Consider an idempotent $f$ of $b$ such that $R:=fbf$ is a basic algebra for $b$. Then we may consider $\L:=fBf$ as a crossed product of $fbf$ with $X$ and $\L$ is Morita equivalent to $B$. Note that $\Pic(b)=\Pic(R)$.

The weak equivalence classes of crossed products of $R$ with $X$ are in 1-1 correspondence with equivalence classes of homomorphisms $\alpha:X \rightarrow \Out(R)$. Hence to determine the possible Morita equivalence classes for $B$ we must determine the equivalence classes of homomophisms $\alpha$.

Now if $\ker(\alpha)=X$, then $\L \cong \cO X \otimes fbf$, Morita equivalent to a product of copies of $fbf$. But this contradicts the fact that $\L$ is Morita equivalent to $B$. By Proposition \ref{Picard}, in each case under consideration $\Pic(b)$ has order divisible only by the primes $2$ and $3$, and so we may suppose that $|X|=3$ and that $\ker(\alpha)=1$.

 In many of the cases we will be making use of the example $PSL_3(7)$ where there is a block which is Morita equivalent to $\cO A_4$ and covered by a nilpotent block of $PGL_3(7)$.

(i) Suppose $b$ is Morita equivalent to $\cO (A_4 \times C_2 \times C_2 )$. By Proposition \ref{Picard} we have $\Pic(b) \cong S_3 \times S_4$ and so there are three possibilities for $\alpha$ up to equivalence (recall that $\alpha$ is assumed to be faithful). The three possible Morita equivalence types for $B$ are given by: $\cO (A_4 \times A_4)$, realised when $N$ is $A_4 \times C_2 \times C_2$; $\cO D$, realised when $G= PGL_3(7) \times C_2 \times C_2$, $N= PSL_3(7) \times C_2 \times C_2$; a non-principal block of $\cO(C_2)^4 \rtimes 3_+^{1+2}$, where the centre of $3_+^{1+2}$ acts trivially, achieved when $N$ is a maximal subgroup of $G=(C_2)^4 \rtimes 3_+^{1+2}$.

(ii) Suppose $b$ is Morita equivalent to the principal block of $\cO(A_5 \times C_2 \times C_2)$. We have $\Pic(b) \cong C_2 \times S_4$ and so there is just one possibility for $\alpha$ up to equivalence, and this is achieved with $G=A_5 \times A_4$.

(iii) Suppose $b$ is Morita equivalent to $\cO(A_4 \times A_4)$. We have $\Pic(b) \cong S_3 \wr C_2$. There are two non-trivial possibilities for $\alpha$ up to equivalence. They give rise to an algebra Morita equivalent to $\cO (A_4 \times C_2 \times C_2)$, realised with $G=A_4 \times PGL_3(7)$, and $\cO((C_2)^4 \rtimes C_3))$ where the $C_3$ acts with only one fixed point. The latter case is realised when $PSL_3(7) \times PSL_3(7) = N < G < PGL_3(7) \times PGL_3(7)$, with $G$ the preimage of the diagonal subgroup of $(PGL_3(7)/PSL_3(7)) \times (PGL_3(7)/PSL_3(7))$.

(iv) Suppose $b$ is Morita equivalent to the principal block of $\cO(A_4 \times A_5)$. We have $\Pic(b) \cong S_3 \times C_2$ and so there is just one possibility for $\alpha$ up to equivalence, and this is realised with $G=PGL_3(7) \times A_5$.

(v)  Finally suppose $b$ is Morita equivalent to the principal block of $\cO(A_5 \times A_5)$. Then $l(b)=9$ and $b$ has a distinguished simple module identified by the unique column of the decomposition matrix of $b$ with all entries equal to $1$, necessarily fixed under the conjugation action of $G$. Write $w:=|G/N|$. Recalling that $B$ is the unique block of $G$ covering $b$, by Clifford theory we have: $l(B)=9w$ if $w \geq 11$; $l(B) \geq 15$ if $w=7$; $l(B) \geq 21$ if $w=5$; $l(B) \in \{ 11,19,27\}$ if $w=3$. By~\cite[Proposition 2.1]{ks13} either $l(B) \leq 9$ or $l(B)=15$, this last case occuring when $B$ has inertial quotient $C_{15}$. Hence we are done unless possibly $w=7$ and $G$ acts with an orbit of length $7$. However, further examination of the decomposition matrix for $b$ reveals that there are precisely two columns with twelve non-zero entries, two with eight non-zero entries and four with four non-zero entries, so such an orbit of length seven is impossible. We have exhausted all possibilities, so conclude that $B$ must be Morita equivalent to $b$ in this case.

\end{proof}

\begin{corollary}
\label{faithfulblocks}
Consider $G=(C_2)^4 \rtimes 3_+^{1+2}$, where the centre of $3_+^{1+2}$ acts trivially. The $2$-blocks of $\cO G$ correspond to the simple modules of $Z(3_+^{1+2})$, and the two non-principal blocks are Morita equivalent. Further, these blocks are Morita equivalent to the two non-principal blocks of $\cO ((C_2)^4 \rtimes 3_-^{1+2})$.
\end{corollary}

\begin{proof}
Let $B$ be any faithful $2$-block of $G=(C_2)^4 \rtimes 3_+^{1+2}$ or $(C_2)^4 \rtimes 3_-^{1+2}$. Then $l(B)=1$. Take a maximal subgroup $N$ of $G$ and a block $b$ of $N$ covered by $B$. Then $N \cong ((C_2)^4 \rtimes C_3) \times C_3$ or $(C_2)^4 \rtimes C_9$ and $b$ is Morita equivalent to $\cO (C_2 \times C_2 \times A_4)$. By Proposition \ref{applying_cross_products} there is only one possibility for the Morita equivalence class of $B$ under the restriction that there is just one simple module.

\end{proof}


\section{Proof of the main theorem}
\label{proofofmaintheorem}

We first address the case where the defect group is normal.

\begin{lemma}
\label{normaldefect}
Let $B$ be a block of $\cO G$ for a finite group $G$ with normal defect group $D \cong (C_2)^4$. Then $B$ is Morita equivalent to a block as in (a) or (b)(i), (ii), (iv), (v), (vi), (viii), (xi) or (xiii) in Theorem \ref{maintheorem}.
\end{lemma}

\begin{proof}
This follows from the main result of~\cite{ku85}, applying Lemma \ref{faithfulblocks} when the inertial quotient is $C_3 \times C_3$.
\end{proof}

For a block $B$, write $\IBr(B)$ for the set of irreducible Brauer characters of $B$ and $l(B)=|\IBr(B)|$.

The following lemma deals for example with the situation $SL_n(q) \cong N \lhd G$ where $G$ is an extension by field automorphisms and the block of $SL_n(q)$ is nilpotent covered.

\begin{lemma}
\label{nilpotentcovered}
Let $G$ be a finite group and $N \lhd G$ such that $G/N$ is solvable. Let $B$ be a quasiprimitive block of $\cO G$ with abelian defect group $D$ covering a block $b$ of $\cO N$ also with defect group $D$. If $b$ is nilpotent covered, then $B$ is Morita equivalent to a block of a finite group with normal defect group. In particular, if $D \cong (C_2)^4$, then $B$ is Morita equivalent to one of the blocks in (a) or (b)(i), (ii), (iv), (v), (vi), (viii), (xi) or (xiii) of Theorem \ref{maintheorem}.
\end{lemma}

\begin{proof}
By Proposition \ref{inertial}(ii) $b$ is inertial, i.e., basic Morita equivalent to its Brauer correspondent $c$ in $N_N(D)$. Let $M$ be the preimage in $G$ of $O_{\ell'}(G/N)$ and $B_M$ the unique block of $M$ covered by $B$. By Proposition \ref{inertial}(iii) $B_M$ is inertial. Write $M_1$ for the preimage in $G$ of $O_\ell(G/M)$ and let $B_{M_1}$ be the unique block of $M_1$ covered by $B$. Note that $B_M$ and $B_{M_1}$ both have defect group $D$. Since $M_1/M$ is an $\ell$-group $B_{M_1}$ is the unique block of $M_1$ covering $B_M$. But then by~\cite[15.1]{alp} $M_1=MD$, and so $M=M_1$. Since $G/N$ is solvable this implies that $M=G$, and $B$ is inertial. The last part follows by Lemma \ref{normaldefect}.
\end{proof}

We prove Theorem \ref{maintheorem}.
\medskip

\begin{proof}
Let $B$ be a block of $\cO G$ for a finite group $G$ with defect group $D \cong (C_2)^4$ with $([G:O_{2'}(Z(G))],|G|)$ minimised in the lexicographic ordering such that $B$ is not Morita equivalent to any of the sixteen blocks listed in the theorem.

Suppose $N \lhd G$ and $b$ is a block of $\cO N$ covered by $B$. Write $I=I_G(b)$ for the stabiliser of $b$ under conjugation. Then there is a unique block $B_I$ of $I$ covering $b$ with Brauer correspondent $B$ (the Fong-Reynolds correspondent) and $B_I$ is Morita equivalent to $B$. Further $B$ and $B_I$ share a defect group, hence by minimality $I=G$. Applying this to all normal subgroups of $G$, we have that $B$ is quasiprimitive, that is, for every $N \lhd G$ each block of $\cO N$ covered by $B$ is $G$-stable.

By Corollary \ref{kpcor} and minimality, if $N \lhd G$ and $B$ covers a nilpotent block of $\cO N$, then $N \leq Z(G) O_2(G)$. In particular $O_{2'}(G) \leq Z(G)$.

Note that $O^2(G)D=G$. This holds by~\cite[15.1]{alp} since any block of $O^2(G)$ covered by $B$ is $G$-stable and $B$ is the unique block of $G$ covering it.

Following~\cite{as00} write $E(G)$ for the \emph{layer} of $G$, that is, the central product of the subnormal quasisimple subgroups of $G$ (the \emph{components}). Write $F(G)$ for the Fitting subgroup, which in our case is $F(G)=Z(G)O_2(G)$. Write $F^*(G)=F(G)E(G) \lhd G$, the generalised Fitting subgroup, and note that $C_G(F^*(G)) \leq F^*(G)$.  Let $b^*$ be the unique block of $\cO F^*(G)$ covered by $B$.

We have $E(G) \neq 1$, since otherwise $F^*(G)=F(G)=Z(G)O_2(G)$ and $D \leq C_G(F^*(G)) \leq F^*(G)$, so that $D \lhd G$, a contradiction by Lemma \ref{normaldefect}. Write $E(G)=L_1 *\cdots *L_t$, where each $L_i$ is a component of $G$ (we have shown that $t \geq 1$). Now $B$ covers a block $b_E$ of $\cO E(G)$ with defect group contained in $D$, and $b_E$ covers a block $b_i$ of $\cO L_i$. Since $b_E$ is $G$-stable, for each $i$ either $L_i \lhd G$ or $L_i$ is in a $G$-orbit in which each corresponding $b_i$ is isomorphic (with equal defect). Since $B$ has defect four, it follows that if $t \geq 3$, then $B$ covers a nilpotent block of a normal subgroup generated by components of $G$, a contradiction. Hence $t \leq 2$, and in particular $G/F^*(G)$ is solvable by the Schreier conjecture.

We have $|F^*(G) \cap D| \geq 4$, since otherwise $B$ covers a nilpotent block of $F^*(G)$, a contradiction since $F^*(G)$ is not central in $G$.

In the next part of the proof we will show that $G$ (as a minimal counterexample) has a proper normal subgroup $N$ containing $D$ such that the unique block $b$ of $N$ covered by $B$ is of type $(C_3)_1$ or $C_3 \times C_3$.

Suppose $|F^*(G) \cap D|=4$. Then $F^*(G) \cap D$ is normal in $N_G(D)$ and so any non-nilpotent block of $O^{2'}(F^*(G) \langle D^g:g \in G \rangle)$ has type $(C_3)_1$, $(C_3)_2$ or $C_3 \times C_3$. We claim that $O^{2'}(F^*(G) \langle D^g:g \in G \rangle)$ is a proper subgroup of $G$. For suppose $O^{2'}(F^*(G) \langle D^g:g \in G \rangle)=G$. Since $O^2(G)=G$, then $F^*(G) \langle D^g:g \in G \rangle = G$. Since $G/F^*(G)$ is solvable it follows that $O^2(G) \neq G$. Since $G=O^2(G)D$, it follows that $G$ has a normal subgroup $H$ of index $2$ containing $F^*(G)$ such that $G=HD$. Hence $B$ must have type $(C_3)_1$ and we may apply Proposition \ref{KK} to show that by minimality $B$ is Morita equivalent to a block on the list. Hence $O^{2'}(F^*(G) \langle D^g:g \in G \rangle)$ is a proper subgroup of $G$ as claimed, and we take $N=O^{2'}(F^*(G) \langle D^g:g \in G \rangle)$. As above $O^2(N) \neq N$ and $N$ has a normal subgroup of index $2$, so that we may rule out the possibilities that $b$ has type $(C_3)_2$ or $C_3 \times C_3$.

Suppose that $|F^*(G) \cap D|=8$. It follows from Proposition \ref{classification16} that one of more of the following occurs: $b^*$ is nilpotent covered; $b^*$ has inertial quotient $C_3$; or $E(G)$ is isomorphic to one of $SL_2(8)$, $^2G_2(3^{2m+1})$, $J_1$ or $Co_3$. In the second case we may take $N=F^*(G) \langle D^g:g \in G \rangle$ and it is clear that $b$ must be of type $(C_3)_1$. Since $O^2(G)=G$ we have $N \neq G$. In the third case, each of the groups $SL_2(8)$, $^2G_2(3^{2m+1})$, $J_1$ and $Co_3$ has odd order outer automorphism group, so $G$ has a direct factor of order $2$, contradicting $O^2(G)=G$. Suppose that $b^*$ is nilpotent covered and does not have inertial quotient $C_3$. By Proposition \ref{inertial} we must have that $b^*$ is inertial and Morita equivalent to $(C_2)^3 \rtimes C_7$. Now $G/F^*(G)$ is solvable, so by Lemma \ref{solvablequotient} $DF^*(G)/F^*(G)$ is a Sylow $2$-subgroup of $G/F^*(G)$, so that $[G:F^*(G)]_2=2$. It follows that $G/F^*(G)$ has a normal $2$-complement. Write $M$ for the preimage in $G$ of this normal $2$-complement and write $B_M$ for the unique block of $\cO M$ covered by $B$, so $B_M$ also covers $b^*$. By Proposition \ref{inertial} $B_M$ is also inertial with abelian inertial quotient. Since $B_M$ has defect group $(C_2)^3$, this inertial quotient must then be cyclic, that is $C_7$ or $C_3$. By~\cite[Corollary 3.7]{mu13} $G$ acts as inner automorphisms on $B_M$, so in particular every simple $B_M$-module is $G$-stable. Hence $l(B)=l(B_M) \in \{ 3,7\}$ and $k(B)=2k(B_M)=16$, so that by~\cite{ks13} $B$ also has cyclic inertial quotient. It follows by Proposition \ref{KK} that $B$ is Morita equivalent to $\cO (D \rtimes C_3)$ or $\cO (D \rtimes C_7)$, contradicting minimality.

Hence we may suppose that $D \leq F^*(G)$. We examine the possibilities for $O_2(G)$.

If $|O_2(G)| = 16$, then $O_2(G)=D$, a contradiction by Lemma \ref{normaldefect}. If $|O_2(G)|=8$, then as $E(G) \neq 1$, $B$ covers a nilpotent block of $E(G)$, a contradiction. If $|O_2(G)| = 4$, then $b^*$ must be of type $(C_3)_1$, $(C_3)_2$ or $C_3 \times C_3$. However $F^*(G)$ would have a normal subgroup of index $2$ and so we may rule out the cases of type $(C_3)_2$ and $C_3 \times C_3$. If $F^*(G)=G$, then $B$ is Morita equivalent to a block in the list, a contradiction. Hence we may take $N=F^*(G)$.

Hence $|O_2(G)| = 1$ or $2$, and $O_2(G) \leq Z(G)$.

Suppose that $t=1$. By Proposition \ref{classification16} one or more of:

(1) $b^*$ has type $(C_3)_1$; or

(2) $F^*(G)$ is isomorphic to one of $C_2 \times SL_2(8)$, $C_2 \times {}^2G_2(3^{2m+1})$, $C_2 \times J_1$, $C_2 \times Co_3$ or $SL_2(16)$, as in each of these cases the component must be simple (since the Schur multiplier is trivial); or

(3) $b^*$ is nilpotent covered.

In case (1) we may take $N=F^*(G)$.

Suppose case (2) occurs. If $F^*(G) \cong C_2 \times J_1$, $C_2 \times Co_3$ or $SL_2(16)$, then $\Out(F^*(G))=1$ and so $G=F^*(G)$. By~\cite{kmn11} the non-principal block of $Co_3$ with elementary abelian defect group of order $8$ is Morita equivalent to the principal block of $\Aut(SL_2(8))$ and so in each of these three cases $B$ is Morita equivalent to a block in the list, a contradiction. If $F^*(G) \cong C_2 \times SL_2(8)$, then $G \cong C_2 \times SL_2(8)$ or $C_2 \times \Aut(SL_2(8))$, again a contradiction. If $F^*(G) \cong C_2 \times {}^2G_2(3^{2m+1})$, then $G$ has $C_2$ as a direct factor and by~\cite[3.1]{ea16} $B$ is Morita equivalent to $b^*$. Hence by minimality $G \cong C_2 \times  {}^2G_2(3^{2m+1})$.  By~\cite[Example 3.3]{ok97}, which in turn uses~\cite{lm80}, $b^*$ is Morita equivalent to the principal block of $C_2 \times ^2G_2(3) \cong C_2 \times \Aut(SL_2(8))$, again a contradiction to minimality.

If (3) occurs, then we may apply Lemma \ref{nilpotentcovered} to obtain a contradiction.

Now suppose that $t=2$. Then $b_1$ and $b_2$ both have Klein four defect group and are non-nilpotent, and $b^*$ has type $C_3 \times C_3$. Hence we may take $N=F^*(G)$.

We have shown that there is a normal subgroup $N \lhd G$ containing $D$ and a block $b$ of $N$ covered by $B$ with type $(C_3)_1$ or $C_3 \times C_3$.

Write $J=G[b] \lhd G$, and let $B_J$ be the unique block of $\cO J$ covering $b$ and covered by $B$. By Proposition \ref{innerautos}(i) $B_J$ is source algebra equivalent to $b$, and so in particular is also of type $(C_3)_1$ or $C_3 \times C_3$. Hence we may assume (repeatedly applying the argument if necessary) that $G[b]=N$. Then by Proposition \ref{innerautos}(ii) $B$ is the unique block of $G$ covering $b$. Hence by~\cite[15.1]{alp} $[G:N]$ is odd since $B$ and $b$ share a defect group, and so $G/N$ is solvable (note that it is not strictly necessary to directly use the odd order theorem here, as in all the cases above $N$ contains $F^*(G)$ and $G/F^*(G)$ is solvable). Let $M \lhd G$ with $[G:M]$ is an odd prime. Let $B_M$ be the unique block of $M$ covered by $B$. By Lemma \ref{typeC3orC3xC3} $B_M$ has type $(C_3)_1$ or $C_3 \times C_3$.

Now by minimality $B_M$ is Morita equivalent to a block as in (a), (b)(ii), (b)(iii), (b)(viii), (b)(ix) or (b)(x) in the statement of the theorem. Suppose that $B_M$ is as in (a). Then $B_M$ is inertial as by minimality there is only one possibility for the Morita equivalence class of $B_M$ and of its Brauer correspondent in $N_M(D)$. So by Proposition \ref{inertial} $B$ is also inertial and by Lemma \ref{normaldefect} is Morita equivalent to one of the listed blocks, a contradiction. We now have that $B_M$ is Morita equivalent to one of the blocks considered in Proposition \ref{applying_cross_products}, which we apply to see that $B$ is one of the blocks listed in the statement of the theorem.

To see that the blocks in cases (a),(b) (i)-(xv) represent distinct Morita equivalence classes it suffices to note that the blocks in case (b) have distinct Cartan matrices and the basic algebras for the blocks in (a) and (b)(i) are not isomorphic.

That the blocks in case (a) cannot be Morita equivalent to a principal block follows from~\cite[6.13]{na} that if the principal block has only one simple module, then it is nilpotent.

Finally, we reference the literature that tells us that representatives of the Morita equivalence classes with the same inertial quotient and number of simple modules are derived equivalent. In the cases below splendid derived equivalences are established between the relevant blocks defined with respect to $k$. Then by~\cite[5.2]{ri96} there is a splendid derived equivalence over $\cO$.

By~\cite[\textsection 3]{ri96} the principal blocks of $kA_4$ and $kA_5$ are splendid derived equivalent. It follows that the blocks in cases (ii) and (iii) are derived equivalent, and that the blocks in cases (viii), (ix) and (x) are derived equivalent. The principal blocks of $kSL_2(16)$ and $k((C_2)^4 \rtimes C_{15})$ (the normalizer of a Sylow $2$-subgroup) are derived equivalent by~\cite{ho01}, and so the blocks in cases (xi) and (xii) are splendid derived equivalent. That the principal blocks of $kJ_1$ and $k((C_2)^3 \rtimes (C_7 \rtimes C_3))$ are derived equivalent follows from~\cite{go97}, and a published proof may be found in~\cite[\textsection 6.2.3]{cr13}, with the observation that the blocks are splendid derived equivalent. Hence the blocks in cases (xiii) and (xiv) are derived equivalent. Finally, the splendid derived equivalence between the blocks in cases (xiii) and (xv) follows from~\cite[Remark 3.4]{ok97} and~\cite[4.33]{cr13}.

\end{proof}

\begin{center} ACKNOWLEDGEMENTS \end{center}

I am deeply indebted to Jon Carlson, who wrote and ran MAGMA routines for calculating outer automorphism groups of the basic algebras defined over $k$, which helped me see their final structure and complete an earlier version of this paper. I also thank Markus Linckelmann for encouraging me to extend my results over $k$ to $\cO$ and for helpful discussions. The papers of Watanabe are essential in completing the classification over $\cO$, and I am indebted to Hu Xueqin for directing me to~\cite{wa05} and to Shigeo Koshitani for showing me~\cite{wa00} and~\cite{zh16}. Finally I thank Cesare Ardito for his careful reading of the manuscript and for his helpful comments, and Michael Livesey for some useful discussions.

\vspace{10mm}
Charles Eaton

School of Mathematics

University of Manchester

Oxford Road

Manchester

M13 9PL

United Kingdom

charles.eaton@manchester.ac.uk
\end{document}